\documentclass[12pt]{amsart}

\usepackage{graphicx}
\usepackage{amsthm}
\swapnumbers

\theoremstyle{plain}
\newtheorem{Thm}{Theorem}

\newtheorem{Lem}[Thm]{Lemma}

\theoremstyle{definition}
\newtheorem{Def}[Thm]{Definition}

\newcommand{\comment}[1]{}

\begin{document}

\title[Stable functions and common stabilizations]{Stable functions and common stabilizations of Heegaard splittings}
\author{Jesse Johnson}
\address{\hskip-\parindent
        Department of Mathematics\\
        Yale University\\
        New Haven, CT 06520\\
        USA}
\email{jessee.johnson@yale.edu}

\thanks{Research supported by NSF MSPRF grant 0602368}

\begin{abstract}
We present a new proof of Reidemeister and Singer's Theorem that any two Heegaard splittings of the same 3-manifold have a common stabilization.  The proof leads to an upper bound on the minimal genus of a common stabilization in terms of the number of negative slope inflection points and type-two cusps in a Rubinstein-Scharlemann graphic for the two splittings.
\end{abstract}

\maketitle

\section{Introduction}
\label{introsect}
A \textit{handlebody} is a 3-manifold homeomorphic to the closure of a regular neighborhood of a connected, finite graph embedded in $\mathbf{R}^3$.  A \textit{Heegaard splitting} for a compact, closed, orientable 3-manifold $M$ is a triple $(\Sigma, H_1, H_2)$ where $\Sigma \subset M$ is a closed, orientable surface and $H_1, H_2 \subset M$ are handlebodies such that $\partial H_1 = \Sigma = \partial H_2 = H_1 \cap H_2$ and $H_1 \cup H_2 = M$.  A \textit{stabilization} of a Heegaard splitting $(\Sigma, H_1, H_2)$ is a new Heegaard splitting constructed by taking a connect sum of $(\Sigma, H_1, H_2)$ with a (standard) Heegaard splitting of $S^3$.  The details of this construction will be described later.

In 1935, Reidemeister~\cite{reid} and Singer~\cite{sing} independently discovered that for any two Heegaard splittings of a given manifold, there is always a third Heegaard splitting that is isotopic to a stabilization of each of the original splittings.  This third Heegaard splitting is called a \textit{common stabilization}.  The \textit{stable genus} of the two original splittings is the smallest possible genus of a common stabilization.

Neither of Reidemeister and Singer's constructions suggest how small one can expect the stable genus to be.  For many examples, there is a common stabilization of genus $p + 1$ where $p$ and $q$ are the genera of the two original splittings, with $p \leq q$.  Rubinstein and Scharlemann~\cite{rub:compar} found a construction for Heegaard splittings of non-Haken manifolds producing a common stabilization of genus at most $5p + 8q - 9$.  This and a quadratic bound for Haken manifolds found later by Rubinstein and Scharlemann~\cite{rub:compar2} are the only known bounds for the stable genus purely in terms of the genera of the original splittings.

In this paper we present a new proof of the existence of common stabilizations.  While this proof does not suggest a bound on the stable genus purely in terms of the genera of the original splittings, it does provide a bound in terms of the original genera plus a somewhat unexpected number.

In their construction of common stabilizations, Rubinstein and Scharlemann~\cite{rub:compar} look at a pair of sweep-outs for the original two Heegaard splittings and define a graph in $\mathbf{R}^2$ which they call the \textit{graphic}.  As Kobayashi and Saeki~\cite{Kob:disc} pointed out, the graphic can be thought of as the image of the discriminant set of a stable function on the complement of the spines of the sweep-outs.  In this paper, rather than looking at sweep-outs representing the Heegaard splittings, we will look at Morse functions.  The two Morse functions define a similar graphic, which is now the image of the discriminant set of a stable function on all of $M$.  

The graphic is the image in $\mathbf{R}^2$ of a smooth immersion with cusps of one or more copies of $S^1$.  We will say that a cusp is type one if a line tangent at the cusp separates the two edges that end at the cusp.  If a line tangent to the cusp does not separate the two edges then the cusp is type two.

We will show that each edge of the graphic can be labeled as either a definite fold edge or an indefinite fold edge.  Each cusp connects an edge of definite fold points to an edge of indefinite fold points.  At each point in the interior of an edge, if we think of the edge as the graph of a function (possibly after rotating the graphic to make it one-to-one), we can define the second derivative of the edge at $p$.  An \textit{inflection point} is one where the second derivative is zero.  (Although the second derivative will change if we rotate the graphic, it will remain zero or non-zero.)  Let $c$ be the number of inflection points with negative slope in the indefinite fold edges (terminology to be defined later) of the graphic.  We will prove the following:

\begin{Thm}
\label{mainthm}
There is a common stabilization of genus less than or equal to $(p + q + c)/2$.
\end{Thm}

The number of negative slope inflection points and type two cusps in the graphic seems at first like a rather arbitrary measure of complexity.  However, it fits into the proof in a very natural way.  Moreover, the number of inflection points and cusps is an indirect measure of the number of crossings in the graphic.  For example, if two edges cross each other $n$ times then between the two edges there will be at least $n - 2$ inflection points.  

Theorem~\ref{mainthm} suggests that in order to get a bound on the stable genus in terms of only $p$ and $q$, one could look for a way to simplify the graphic 
for any two Morse functions until $c$ is bounded.  This appears to be no simple task.  In particular, the number of crossings in the graphic is related to the pants distance of each of the two Heegaard splittings (see~\cite{johnson}) and therefore can be arbitrarily high.

The proof of Theorem~\ref{mainthm} relies on the analogy between Heegaard splittings and Morse functions, using a method similar to Hatcher and Thurston's construction~\cite{hatcherthurston} of a presentation for the mapping class group of a surface:  We replace the two Heegaard splittings with two Morse functions, then connect them by a generic path in $C^\infty(M)$.  This path passes through a finite number of near Morse functions.  At these points, the induced Heegaard splittings either don't change, or change in a simple way which is exactly stabilization or destabilization.

When the original two Morse functions are in general position, a straight line connecting them in $C^\infty(M)$ will be a generic path.  Edelsbrunner~\cite{edels} pointed out that the critical points of the intermediate functions are all in the discriminant set of the function $f \times g : M \rightarrow \mathbf{R}^2$.  Because the Rubinstein-Scharlemann  graphic is related to the discriminant set of $f \times g$, this allows us to read information about this path from the graphic.  We will show that when the path passes through a non-Morse function, it induces a stabilization or destabilization corresponding to an inflection point with negative slope in an edge of the graphic or a type-two cusp with negative slope.  This leads to the proof of Theorem~\ref{mainthm}.

Stabilization is described in more detail in Section~\ref{stabsect}.  The connection between Heegaard splittings and Morse functions is discussed in Section~\ref{morsesect}, then Section~\ref{graphicsect} introduces stable function and the Rubinstein-Scharlemann graphic.  In Section~\ref{reebsect}, the singularities of stable functions from 3-manifolds to $\mathbf{R}^2$ are described by looking at Stein filtrations of the functions.  The connection between graphics and stabilizations is described in Section~\ref{linpathsect}, leading to the proof of Theorem~\ref{mainthm} in Section~\ref{proofsect}.  I want to thank Abby Thompson for pointing me in the direction that led to this proof.

\section{Stabilization}
\label{stabsect}

In this section we describe the correspondence between Morse functions and Heegaard splittings that is the basis for the rest of the paper.  Recall that a handlebody is a manifold homeomorphic to the closure of a regular neighborhood of a connected, finite, embedded graph in $\mathbf{R}^3$.  Such a manifold can be thought of as the result of attaching a number of 1-handles to 0-handles in a way that produces a compact, connected and orientable manifold.  

A \textit{compression body} $H$ is a connected, orientable 3-manifold that results from attaching a number of 1-handles to 0-handles and to the $F \times \{0\}$ boundary of a manifold $F \times [0,1]$, where $F$ is a compact closed, not necessarily connected surface with no sphere components.  The union of the boundary components of $H$ coming from $F \times \{1\}$ are written $\partial_- H$ and the remaining component is $\partial_+ H$.  When $F$ is empty, the compression body $H$ is a handlebody with $\partial_+ H = \partial H$ and $\partial_- H = \emptyset$.

In Section~\ref{introsect}, we defined Heegaard splittings for closed manifolds.  
For a compact, connected, orientable 3-manifold $M$ with boundary, a \textit{Heegaard splitting} is a triple $(\Sigma, H_1, H_2)$ where $\Sigma$ is a compact, closed, orientable surface and $H_1$ and $H_2$ are compression bodies such that $\partial_+ H_1 = \Sigma = \partial_+ H_2$ and $\partial M = \partial_- H_1 \cup \partial_- H_2$.  Moreover, the union of $H_1$ and $H_2$ must be all of $M$ and their intersection must be precisely $\Sigma$.

Recall that given 3-manifolds $M_1$ and $M_2$, the \textit{connect sum} $M_1  \# M_2$ is the result of removing an open ball from each manifold and gluing together the resulting spherical boundary components.  Given a Heegaard splitting for each manifold, if we choose the open ball in the manifold to intersect the Heegaard surface in an open disk, then we can glue the manifolds so as to induce a Heegaard splitting on $M_1 \# M_2$.  If $M_2 \cong S^3$ then $M_1 \# M_2 \cong M_1$ and the Heegaard splitting coming from $M_1 \# M_2$ is called a \textit{stabilization}.  The original Heegaard splitting will be called a \textit{destabilization} of the new one.

The connect sum can, of course, be taken along any open disk in the original Heegaard splitting and with any Heegaard splitting of $S^3$.  However, because Heegaard splittings of $S^3$ are standard (see~\cite{wald}, or a number of more recent proofs) and any two open disks in a surface are isotopic, we get the following result.  (Details of the proof are left to the reader.)

\begin{Lem}
\label{stabisstablem}
Let $(\Sigma, H_1, H_2)$ be a Heegaard splitting.  Any stabilization of a stabilization of $(\Sigma, H_1, H_2)$ is isotopic to a stabilization of $(\Sigma, H_1, H_2)$.  Any two stabilizations of $(\Sigma, H_1, H_2)$ are isotopic if and only if they have the same genus.
\end{Lem}

As described in Section~\ref{introsect}, our goal is to show that any two Heegaard splittings of the same manifold have a common stabilization.  Previous proofs have done this by directly constructing this common stabilization.  In this proof, we will go about it rather indirectly.  

Note that a Heegaard splitting $(\Sigma, H_1, H_2)$ is determined entirely by the Heegaard surface $\Sigma$, up to labeling of the handlebodies.  In the following discussion, we will refer to the Heegaard splittings by their Heegaard surfaces in order to avoid overly complicated notation.  Consider a sequence $\Sigma_0,\dots,\Sigma_c$ of Heegaard surfaces such that for each $i \leq c$, $\Sigma_{i+1}$ is a single stabilization of $\Sigma_i$ (a connect sum with a genus one Heegaard splitting of $S^3$) or a single destabilization of $\Sigma_{i+1}$.  

The genera of the splittings in the sequence from $\Sigma_1$ to $\Sigma_c$ go up and down as we pass through the sequence.  If we can find a sequence $\Sigma_0,\dots,\Sigma_c$ such that the genus only increases from $\Sigma_0$ to $\Sigma_d$ for some $d$, then only decreases from $\Sigma_d$ to $\Sigma_c$ then the first half of Lemma~\ref{stabisstablem} tells us that $\Sigma_d$ is a common stabilization for $\Sigma_0$ and $\Sigma_c$.  The second half of Lemma~\ref{stabisstablem} allows us to throw away the condition that the genera of the surfaces first increase, then decrease.  In particular, it implies the following:

\begin{Lem}
\label{ifcstabslem}
If there is a sequence $\Sigma_0,\dots, \Sigma_c$ of single stabilizations and destabilizations then $\Sigma_0$ and $\Sigma_c$ have a common stabilization of genus $(p + q + c) / 2$.
\end{Lem}

Note that $c \equiv (p + q)$ mod $2$ and $c \geq |p-q|$ so $(p + q + c) / 2$ is always an integer no less than $p$ or $q$.

\begin{proof}
If there is a common stabilization of genus $g$ then there is a common stabilization of genus $g + h$ for any positive integer $h$.  Thus we need only show that there is a common stabilization of genus less than or equal to $(p + q + c) / 2$.  Let $\Sigma_0,\dots,\Sigma_c$ be a sequence single stabilizations and destabilizations.  

If $\Sigma_i$ is a destabilization of $\Sigma_{i-1}$ and $\Sigma_{i+1}$ is a stabilization of $\Sigma_i$ (i.e. the genus decreases, then increases) then by definition, both $\Sigma_{i-1}$ and $\Sigma_{i+1}$ are (single) stabilizations of $\Sigma_i$.  By Lemma~\ref{stabisstablem}, this implies $\Sigma_{i+1}$ is isotopic to $\Sigma_{i-1}$.  Thus we can remove $\Sigma_i$ and $\Sigma_{i+1}$ from the sequence, then renumber to get a new, shorter sequence $\Sigma_0,\dots,\Sigma_{c'}$ of single stabilizations and destabilizations with isotopic starting and ending surfaces.

By removing any extra destabilization-stabilization pairs in this way, we can replace the original sequence of surfaces with a possibly shorter sequence in which the genera increase from $\Sigma_0$ to some $\Sigma_d$, then decrease from $\Sigma_d$ to $\Sigma_{c'}$ (with $c' \leq c$).  As noted above, $\Sigma_d$ is a common stabilization of $\Sigma_0$ and $\Sigma_{c'}$ (which is isotopic to $\Sigma_c$).  One can check that $\Sigma_d$ has genus at most $(p + q + c') /2 \leq (p + q + c) /2$, completing the proof.
\end{proof}

In the following sections, we will construct a sequence of stabilizations and destabilizations as above for an arbitrary pair of Heegaard splittings.

\section{Morse functions}
\label{morsesect}

Recall that a Morse function on a smooth manifold $M$ is an infinitely differentiable function $f : M \rightarrow \mathbf{R}$ (i.e. $f$ is in $C^\infty(M,\mathbf{R})$) with certain properties.  We will now review these properties in detail.

Recall that given a point $p$ in a smooth manifold $M$ and a smooth function $f \in C^\infty(M, \mathbf{R})$, the \textit{gradient} of $f$ at $p$ is the vector defined by the partial derivatives of $f$ at $p$.   The point $p$ is a \textit{critical point} of $f$ if the gradient of $f$ at $p$ is zero.  The \textit{Hessian} of $f$ at $p$ is the matrix of second derivatives of $f$ at $p$.  A critical point is \textit{degenerate} if the determinant of the Hessian is zero, and \textit{non-degenerate} otherwise (see~\cite{milnor} for details). 

The behavior of a function near a non-degenerate critical point was classified by Morse.  In a 3-dimensional manifold, there are four types of non-degenerate critical points.  The behavior of the level sets of $f$ in neighborhoods of these four types of points is shown in Figure~\ref{crits3fig}.
\begin{figure}[htb]
  \begin{center}
  \includegraphics[width=2in]{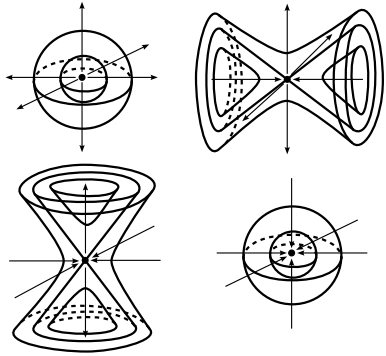}
  \caption{The four types of critical points of a 3-dimensional Morse function.}
  \label{crits3fig}
  \end{center}
\end{figure}

Near these points, $f$ is modeled by $f(x,y,z) = x^2 + y^2 + z^2$, $f(x,y,z) = x^2 + y^2 - z^2$, $f(x,y,z) = x^2 - y^2 - z^2$ or $f(x,y,z) = - x^2 - y^2 - z^2$.  Each critical point is said to have \textit{index} equal to the number of negative signs in the local description.  Thus the top two critical points shown in Figure~\ref{crits3fig} have index zero and one, while the bottom two have index two and three.

Note that there is an open neighborhood around each non-degenerate critical point such that it is the only critical point in that neighborhood.  The \textit{level} of a critical point $p \in M$ is simply $f(p)$.  If $\ell \in \mathbf{R}$ is the level of a critical point then $\ell$ is called a \textit{critical level}.  Otherwise, $\ell$ is called a \textit{regular level}.

\begin{Def}
A \textit{Morse function} is a smooth function such that (1) every critical point is non-degenerate and (2) no two critical points are at the same level.
\end{Def}

Each non-degenerate critical point is contained in an open neighborhood containing no other critical points.  Thus there is an open cover of $M$ such that each open set contains at most one critical point.  If $M$ is compact then the finite sub-cover property implies that there are finitely many critical points, and finitely many critical levels.

If $M$ is a manifold with boundary, a Morse function on the interior of $M$ is called \textit{proper} if in some neighborhood of $\partial M$, the level sets of $f$ consist entirely of boundary parallel surfaces, and $f$ extends (uniquely) to $\partial M$.  Given a proper Morse function $f : M \rightarrow \mathbf{R}$ and an interval $[a,b] \subset \mathbf{R}$ such that $a$ and $b$ are regular levels, the restriction of $f$ to the submanifold $f^{-1}[a,b] \subset M$ will be a proper Morse function on $f^{-1}[a,b]$.

The connection between Heegaard splittings and Morse functions is a result of the following Lemma:

\begin{Lem}
\label{morsehandlem}
Let $M$ be a compact, orientable 3-manifold.  If there is a proper Morse function $f : M \rightarrow \mathbf{R}$ with only index zero and index one critical points then every component of $M$ is a compression body.
\end{Lem}

The lemma can be deduced from the classification of Morse singularities described above.  The details are left to the reader.  The equivalent lemma for index two and three critical points holds for the same reasons.  

If a component of $M$ has connected boundary then this component is in fact a handlebody.  In this case, if there are $n$ index zero critical points and $m$ index one critical points in the component then the handlebody has genus $m - n + 1$.  (Again, details are left to the reader.)

Given a Morse function $f$ on a 3-manifold $M$, assume there is a value $b \in \mathbf{R}$ such that every index zero or one critical point is below $b$, while every index two or three critical point is above $b$.  The submanifold $H_1 = f^{-1}((-\infty,b])$ contains only index zero and one critical points and is thus a compression body.  Similarly, $H_2 = f^{-1}([b,\infty))$ is a second compression body.  If we define $\Sigma = f^{-1}(b)$ then we have $\Sigma = \partial_+ H_1 = \partial_+ H_2$ so $(\Sigma, H_1, H_2)$ is a Heegaard splitting for $M$.

Conversely to Lemma~\ref{morsehandlem}, given a Heegaard splitting $(\Sigma,H_1,H_2)$ of a manifold $M$, one can construct a Morse function on each handlebody consisting of only index zero and one critical points or only index two and three critical points, respectively.  Moreover, one can construct these function such that they agree on $\Sigma$, inducing a Morse function on $M$ in which $\Sigma$ is a level set.

An arbitrary Morse function on $M$ will not have this important property that there is a regular level separating the high index critical points from the low index critical point.  However, for such a function, one can choose a finite set of levels $\ell_1,\dots,\ell_{2n+1}$ such that for each $i$, the interval $[\ell_{2i-1},\ell_{2i}]$ contains only index zero and one critical points and each interval $[\ell_{2i},\ell_{2i+1}]$ contains only index index two and three critical points.  

The surfaces $\bigcup f^{-1}(\ell_i)$ cut $M$ into a collection of compression bodies, defining a structure called a \textit{generalized Heegaard splitting}.  Schultens~\cite{schultens} showed that such a structure can be turned into a unique (up to isotopy) Heegaard splitting by a process called \textit{amalgamation}.  We will not review the construction here.  The key is that a Morse function determines a unique isotopy class of generalized Heegaard splittings, which in turn determines a unique isotopy class of Heegaard splittings.  This we have the following:

\begin{Lem}
\label{morseheedlem}
Every Morse function on $M$ determines a unique (up to isotopy) Heegaard splitting $(\Sigma, H_1, H_2)$ on $M$.  If $M$ is closed then the genus of $\Sigma$ is $m - n + 1$ where $m$ is the number of index one critical points and $n$ is the number of index zero critical points.
\end{Lem}

We will restrict our attention to the closed case so the genus of the induced Heegaard splitting will always be $m - n + 1$.  

A \textit{spine} for a handlebody $H$ is a graph $K \subset H$ such that the complement $H \setminus K$ is homeomorphic to $\partial H \times (0,1]$.  If $H$ is embedded in a 3-manifold then $H$ is isotopic to a regular neighborhood of $K$.  Thus $H$ is determined, up to isotopy, entirely by its spine.  If $H$ is a handlebody in a Heegaard splitting then the Heegaard surface (the boundary of $H$) is determined by $K$, so the entire Heegaard splitting is determined, up to isotopy, entirely by a spine for one of its handlebodies.  Thus we will try to understand the Heegaard splitting induced by a Morse function by constructing a spine for one of the handlebodies.

For an index one critical point $p$ of a Morse function $f$, a \textit{descending arc} is an arc $\alpha : [0,1] \rightarrow M$ such that $\alpha(0) = p$, $\alpha(1)$ is an index zero critical point and the function $f \circ \alpha : [0,1] \rightarrow \mathbf{R}$ is monotonically decreasing.  

For each of the index one critical points of $f$, there are a number of different descending arcs.  We will pick a pair of descending arcs for each critical point that approach the critical points from opposite directions.  Let $K$ be the union of these pairs of descending arcs.  We will call this graph $K$ a \textit{descending spine}.  There are many descending spines for a Morse function, however the construction of an amalgamation implies the following connection between a descending spine and the induced Heegaard splitting.  Details of the proof are left to the reader.

\begin{Lem}
\label{kisaspinelem}
The graph $K$ is isotopic to the spine of a handlebody in the Heegaard splitting defined by $f$.
\end{Lem}

\section{The Rubinstein-Scharlemann Graphic}
\label{graphicsect}

Let $\phi, \psi : X \rightarrow Y$ be smooth maps between smooth manifolds.  We will say that $\phi$ and $\psi$ are \textit{isotopic} if there are automorphisms $A_X : X \rightarrow X$ and $A_Y : Y \rightarrow Y$, each isotopic to the identity on its respective space, such that $\phi = A_Y \circ \psi \circ A_X$.  The function $\phi$ will be called \textit{stable} if there is an open neighborhood $N \subset C^\infty(X,Y)$ (under the Whitney $C^\infty$ topology, see~\cite{golub} or~\cite{dupl}) around $\phi$ such that every map in $N$ is isotopic to $\phi$.  In other words, small perturbations of a stable map do not change its topology.

If a map $\psi$ is isotopic to a stable map $\phi$ then the isotopies induce an automorphism of $C^\infty(X,Y)$, preserving the norm, so $\psi$ is also stable.  Thus the open ball $N$ consists of stable maps, implying that the set of stable maps in $C^\infty(X,Y)$ is an open set.  If two stable maps are connected by an arc $\alpha$ of stable maps then $\alpha$ is covered by a finite collection of open sets such that any two maps in each set are isotopic.  By induction, any two maps in $\alpha$ are isotopic.  Thus each path component of the set of stable maps represents a single homeomorphism/isotopy class of smooth maps.

A stable map from a manifold to $\mathbf{R}$ is simply a Morse function.  Let $f$ and $g$ be Morse function on $M$.  The product of $f$ and $g$ is a map $F = f \times g : M \rightarrow \mathbf{R}^2$.  (Define $F(x,y) = (f \times g)(x,y) = (f(x),f(y))$.)  We can recover $f$ and $g$ from $F$ by projecting onto the horizontal and vertical axes of $\mathbf{R}^2$.  In other words, $f = p_x \circ F$ where $p_x$ is the orthogonal projection map from $\mathbf{R}^2$ onto $\mathbf{R} \times \{0\}$ and $g = p_y \circ F$ where $p_y$ is the orthogonal projection map from $\mathbf{R}^2$ onto $\{0\} \times \mathbf{R}$.

For a compact, closed, orientable 3-manifold $M$, Mather~\cite{mather} showed that the set of stable maps in $C^\infty(M, \mathbf{R}^2)$ is a dense set.  (He also showed this for a number of other dimensions.)  Thus any open neighborhood of $f \times g$ contains a stable function.  

The projection maps $p_x$ and $p_y$ define continuous maps from $C^\infty(M, \mathbf{R}^2)$ to $C^\infty(M, \mathbf{R})$ (See~\cite[Ch. 2, Proposition 3.5]{golub}.   Because $f$ is stable (Morse), there is an open neighborhood $N_f$ of $f$ in $C^\infty(M,\mathbf{R})$ in which all the functions are isotopic to $f$.  The pre-image of $N_f$ in the map induced by $p_x$ is an open set in $C^\infty(M,\mathbf{R}^2)$.  Similarly, the pre-image of an open neighborhood of $g$ is an open set in $C^\infty(M, \mathbf{R}^2)$.  The intersection of these two open sets is open so the intersection contains a stable map $F'$.

The composition of $F'$ with the projection $p_x$ is a function $f'$ isotopic to $f$.  The composition of $F'$ with $p_y$ is a function $g'$ isotopic to $g$.  Thus $F' = f' \times g'$ where $f'$ and $g'$ are (Morse) functions isotopic to $f$ and $g$, respectively.  If we isotope $f$ and $g$ to $f'$ and $g'$ then the product of $f$ and $g$ will be a stable map.  In other words, we have proved the following:

\begin{Lem}
If $f$ and $g$ are Morse functions then after arbitrarily small isotopies of $f$ and $g$, the product $F = f \times g$ will be a stable map.
\end{Lem}

Assume $F = f \times g$ is a stable map.  The \textit{Jacobi set} or \textit{discriminant set} $\mathcal{J}$ of $F$ is the set of points $p \in M$ where the discriminant map $T_p M \rightarrow \mathbf{R}^2$ has a two dimensional kernel.  (At the remaining points, this map will have a one dimensional kernel.)  

In terms of $f$ and $g$, $\mathcal{J}$ is the set of points $p$ where the gradients of $f$ and $g$ are linearly dependent in $T_p M$.  In other words, the gradients at $p$ are parallel or one of the gradients is zero.  (In the latter case, $p$ is a critical point of $f$ or $g$.)  Equivalently, $\mathcal{J}$ is the set of critical points of $f$ and $g$ and points in $M$ where the level surfaces of $f$ and $g$ are tangent.  The image in $f \times g$ of $\mathcal{J}$ is a one dimensional set in $\mathbf{R}^2$ which we will call the \textit{graphic}.  We will think of the graphic as drawn so that $f$ increases from left to right, while $g$ increases from bottom to top.

Rubinstein and Scharlemann defined the graphic slightly differently, beginning with sweep-outs rather than Morse functions.  From their point of view, a sweep-out is a family of parallel surfaces that fill a manifold, expanding out of one spine of the Heegaard splitting and collapsing onto the other.  By employing results of Cerf~\cite{cerf:strat}, they define a general position for two sweep-outs and define the graphic as the set of points where leaves of the two sweep-outs are tangent.

Our third description of the graphic given above as the image in $\mathbf{R}^2$ of the points where level surfaces are tangent should seem very reminiscent of  Rubinstein and Scharlemann's definition.  This is intentional.  Kobayashi and Saeki~\cite{Kob:disc} showed that by thinking of a sweep-out as a function from $M$ to $\mathbf{R}$ (the family of surfaces become level sets of this function), Rubinstein and Scharlemann's definition can be thought of as the image of the discriminant of a stable function on the complement in $M$ of the spines.  Our viewpoint takes this one step further, replacing the sweep-out functions with Morse functions and producing a stable function on all of $M$.

Above, we used the fact that we can recover $f$ and $g$ from the stable function $F$ by composing $F$ with projections onto the horizontal and vertical axes, respectively.  We can construct other functions by composing $f$ with projection onto an arbitrary line $L$ through the origin in $\mathbf{R}^2$.  Such a projection can be written as a linear combination $p_L = a p_x + b p_y$ where $a,b \in \mathbf{R}$ are determined by the slope of $L$.  The composition of $F$ with this linear combination of $p_x$ and $p_y$ is in turn a linear combination of $f$ and $g$, namely $a f + b g$.  

If we think of $f$ and $g$ as points in the vector space $C^\infty(M,\mathbf{R})$ then the different projections of $F$ determine points of the plane in $C^\infty(M,\mathbf{R})$ spanned by vectors $f$ and $g$.  If we choose the projections given by coefficients $a = \sin(t)$, $b = \cos(t)$ then the family of projections determine the arc $\alpha = \{\sin(t) f + \cos(t) g\ |\ t \in [0,\frac{\pi}{2}]\}$ from $f$ to $g$ in $C^\infty(M, \mathbf{R})$.

The intersection of $\alpha$ with the set of Morse functions in $C^\infty(M, \mathbf{R})$ is open in $\alpha$ because the set of Morse functions is open in $C^\infty(M, \mathbf{R})$.  Each component of the intersection determines a single isotopy class of Morse functions so if there are finite number of components then the arc $\alpha$ determines a finite sequence of Heegaard splittings.  We will see that ``generically'', the number of components is in fact finite.

In order to prove Theorem~\ref{mainthm}, we must show two things:  first, that the sequence of Heegaard splittings determined by $\alpha$ is a sequence of single stabilizations and destabilizations and second, that the number of stabilizations and destabilizations in this sequence is bounded by the number of negative slope inflection points in the graphic.  In order to do both of these, we must understand how the topology of the graphic corresponds to the topology of the Morse functions determined by projections of $F$.

\section{The Reeb Complex and the Stein Filtration}
\label{reebsect}
Our main tool for interpreting the graphic will be a 2-complex through which we will filter the map $f \times g : M \rightarrow \mathbf{R}^2$.  However, before introducing this complex we will introduce a related tool one dimension lower.

Given a compact, closed, orientable surface $\Sigma$, let $f : \Sigma \rightarrow \mathbf{R}$ be a Morse function on $\Sigma$.  Define the equivalence relation $\sim$ on points in $\Sigma$ by $x \sim y$ whenever $x,y \in M$ are in the same component of a level set of $f$.  The \textit{Reeb graph} is the quotient of $\Sigma$ by the relation $\sim$.  

As suggested by the name, the Reeb graph $G = \Sigma / \sim$ is a graph.  The edges of $G$ come from annuli in $\Sigma$ fibered by level loops.  The vertices correspond to critical points of $f$, with valence one vertices corresponding to central singularities and valence three vertices corresponding to saddle singularities, as in Figure~\ref{locreebfig}.  A simple Euler characteristic argument shows that the rank of the fundamental group of $G$ is equal to the genus of $\Sigma$.  (In fact, $G$ is isomorphic to a spine for a handlebody bounded by $\Sigma$.)
\begin{figure}[htb]
  \begin{center}
  \includegraphics[width=2in]{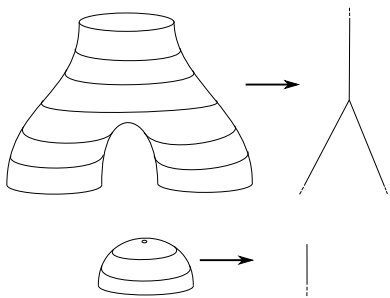}
  \caption{In the Reeb graph, valence three vertices correspond to saddle singularities and valence one vertices correspond to central singularities.}
  \label{locreebfig}
  \end{center}
\end{figure}

There is a map from $\Sigma$ to $G$ induced by the quotient.  The pre-image in the quotient map of each point in $G$ is a component of a level set, so there is also an induced map from $G$ to $\mathbf{R}$.  The composition of these two maps $\Sigma \rightarrow G \rightarrow \mathbf{R}$ is precisely $f$.  Although the quotient in the two dimensional case has the name Reeb attached to it, in general the method of defining a quotient space and writing a stable map as a composition of two maps in this way is called a \textit{Stein factorization} (See~\cite{shadows}).  In an attempt to avoid the politics of such a naming conflict, we will attach the name Reeb to the quotient space, and attach the name Stein to the pair of maps whose composition is the original stable map.

Given a compact, closed, orientable 3-manifold $M$ and a stable function $F : M \rightarrow \mathbf{R}^2$, define the equivalence relation $\sim$ on $M$ by $x \sim y$ whenever $x,y \in M$ are in the same component of a pre-image of a point in $\mathbf{R}^2$.  As in the two dimensional case, there is a Stein filtration from $M$ to the \textit{Reeb complex} $C = M / \sim$ and from $C$ to $\mathbf{R}^2$ such that the composition $M \rightarrow C \rightarrow \mathbf{R}^2$ is $F$.

Let $D \subset \mathbf{R}^2$ be an open disk disjoint from the image of the discriminant set $\mathcal{J}$ in $\mathbf{R}^2$ (i.e. the graphic).  The pre-image in $M$ of $D$ is a collection of solid tori such that level sets of $F$ foliate these solid tori by longitudes.  Quotienting $f^{-1}(D)$ by $\sim$ sends each solid torus to an open disk in $C$.  (The map from $C$ to $\mathbf{R}^2$ is one-to-one on each disk.)  

Thus a large portion of $C$ consists entirely of disks.  We would like to show that $C$ is in fact homeomorphic to a two dimensional cell complex.  To do this, we must examine the local structure of $C$ near $K$.  Mather's~\cite{mather} classification of critical points of stable maps into $\mathbf{R}^2$ (See also~\cite{levine}) implies the following:

\begin{Thm}[Mather]
\label{critsclassthm}
If $F : M \rightarrow \mathbf{R}^2$ is a stable map (where $M$ is a closed, orientable 3-manifold) then at any critical point $p \in M$, there is an open neighborhood of $p$ that can be parametrized with coordinates $u,x,y$ so that for some parametrization of $\mathbf{R}^2$, $F(u,x,y)$ has one of the following forms:

(1) $F(u,x,y) = (u, x^2 + y^2)$ ($p$ a definite fold point),

(2) $F(u,x,y) = (u, x^2 - y^2)$ ($p$ an indefinite fold point),

(3) $F(u,x,y) = (u,y^2 + ux - \frac{x^3}{3})$ ($p$ a cusp point).

Moreover, no cusp point is a double point of the map from $\mathcal{J}$ to $\mathbf{R}^2$ and on the complement of the cusps, the map from $\mathcal{J}$ to $\mathbf{R}^2$ is an immersion with normal crossings.
\end{Thm}

For each type of critical point, we can think of $F$ as the product of two functions from $N$ to $\mathbf{R}$.  The discriminant set is the set of points where the gradients of the two functions agree.  In the models of all three critical points, the first function is simply $a(u,x,y) = u$.  Thus the gradient at each point is the vector $(1,0,0)$.  The gradient of the second function will be parallel to $(1,0,0)$ if and only if the derivatives in the $x$ and $y$ directions are zero.  Thus the discriminant set in the local patch $N$ is given by the equations $\frac{db}{dx} = 0$ and $\frac{db}{dy} = 0$ where $b(u,x,y)$ is $x^2 + y^2$, $x^2 - y^2$ or $y^2 + ux - \frac{x^3}{3}$, respectively.  Note that the functions $a$ and $b$ will not, in general, be equal to $f$ and $g$ because in order to get the form shown in the theorem, we must reparametrize $\mathbf{R}^2$.

At the first type of critical point, a definite fold point, the discriminant set intersects $N$ in the arc $\{(t, 0, 0)\}$ and maps to the arc $\{(t,0)\}$ in $\mathbf{R}^2$.  Locally, the pre-image of each point in $(t,0)$ is the single point $(t,0,0)$ in $N$.  The pre-image of a nearby point $(t,\epsilon)$ is a loop around $(t,0,0)$ as shown in Figure~\ref{deffoldfig}.  The quotient of $N$ by $\sim$ is a disk whose boundary consists of an arc in the image $K$ of $\mathcal{J}$ and an arc disjoint from $K$.  The induced map from this disk into $\mathbf{R}^2$ is one-to-one.
\begin{figure}[htb]
  \begin{center}
  \includegraphics[width=3.5in]{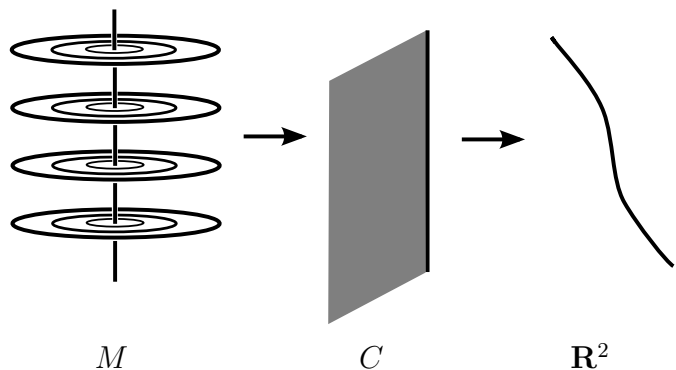}
  \put(-220,-15){$M$}
  \put(-120,-15){$C$}
  \put(-40,-15){$\mathbf{R}^2$}
  \caption{The local maps from $M$ to $C$ to $\mathbf{R}^2$ at a definite fold point.}
  \label{deffoldfig}
  \end{center}
\end{figure}

In order to put $F$ in the form shown in Theorem~\ref{critsclassthm} it is necessary to reparametrize $\mathbf{R}^2$.  Thus although the edge of $K$ maps to a vertical arc in $\mathbf{R}^2$ in this form, in general it will map to an arbitrary smooth arc.  However, the local structure of $C$ will be the same and the map from the disk neighborhood of the edge into $\mathbf{R}^2$ will be locally one-to-one.

At the second type of critical point, an indefinite fold, the discriminant set is again the vertical arc $\{(t,0,0)\}$ in $N$, whose image in $\mathbf{R}^2$ is the arc $\{(t,0)\}$.  The pre-image of a nearby point $(t,\epsilon)$ or $(t, -\epsilon)$ is a pair of arcs, each of which sits in a level loop of $F$.  In the quotient, these two arcs will map to separate points of $C$ if they are in different level loops, or the same point of $C$ if they are in the same level loop.  

If $(t,0)$ is not a double point in the graphic then the non-loop component of the preimage in $M$ of $(t,0)$ will be a figure eight, i.e. a graph with two edges and a single valence four vertex at $(t,0,0)$.  The boundary of a neighborhood of this figure eight consists of three loops.  Thus the arcs in $N$ that are the pre-image of $(t,\epsilon)$ will sit in the same level loop of $F$ if and only if the arcs in the pre-image of $(t,-\epsilon)$ are in different level loops.  This implies that if there are no double points in the image of $N \cap \mathcal{J}$ in the graphic then the image in $C$ of $N \cap \mathcal{J}$ is a valence three edge such that two faces enter the edge from one side and one face enters from the other side, as in Figure~\ref{indeffoldfig}.  Once again, the image of $N \cap \mathcal{J}$ will not, in general be a vertical arc.  It only appears vertical in the local reparametrization.
\begin{figure}[htb]
  \begin{center}
  \includegraphics[width=3.5in]{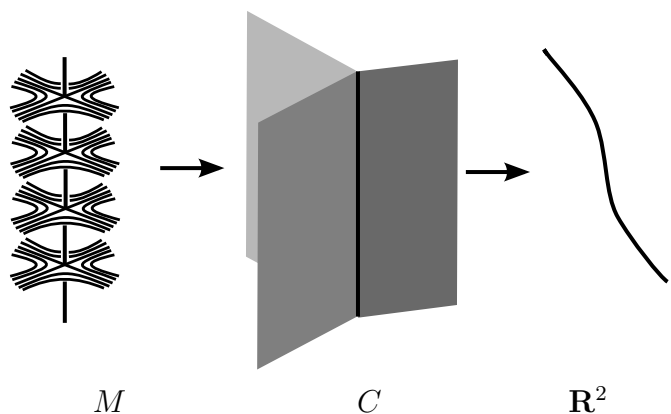}
  \put(-220,-15){$M$}
  \put(-120,-15){$C$}
  \put(-40,-15){$\mathbf{R}^2$}
  \caption{The local maps from $M$ to $C$ to $\mathbf{R}^2$ at an indefinite fold point.}
  \label{indeffoldfig}
  \end{center}
\end{figure}

If $N \cap \mathcal{J}$ contains a double point of the graphic then things are slightly more complicated.  There are a finite number of double points in the graphic so $N$ can be chosen so that its image in $\mathbf{R}^2$ contains exactly one of them.  The non-loop component or components of the pre-image of this double point in $M$ form a graph consisting of four edges and two valence four vertices.  If the graph is disconnected (i.e. consists of two figure eights) then the local behavior of $C$ is as in the non-double point case, but the arcs of $\mathcal{J}$ that contain the two critical points are sent to arcs in $\mathbf{R}^2$ that happen to cross.  Bachman and Schleimer~\cite{bachmanschleimer} call this an \textit{unentangled crossing}.

If the graph determined by the level set at the crossing is connected then the two arcs of $\mathcal{J}$ containing the critical points at the crossing are sent to arcs that cross in $C$ (as well as in $\mathbf{R}^2$).  The vertex in the graphic is called an \textit{entangled crossing}.  Note that the two arcs are still disjoint in $M$.  In $C$, we get a valence four vertex adjacent to six two-cells.  The possible ways that these six faces can come together at a vertex are not important for the proof, but can be worked out by the reader.

In the final type of critical point, a cusp point, the intersection of $\mathcal{J}$ with $N$ is given by the equations $u = x^2$, $y = 0$.  In order to understand the topology, we note that the level sets of $b$ are as in Figure~\ref{cuspfig}.  In $N$ these level sets sit on top of each other, but they are drawn in this way to avoid an overly confused picture.  The loops of intersection between the level sets of $b$ and those of $a$ are shown.
\begin{figure}[htb]
  \begin{center}
  \includegraphics[width=3.5in]{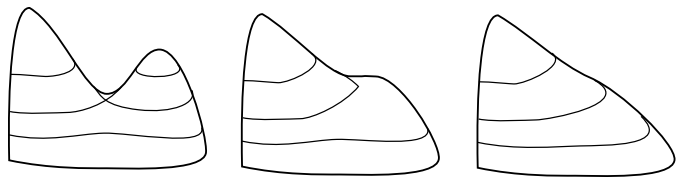}
  \caption{A saddle and central singularity cancel each other at a degenerate critical point, forming a cusp or birth/death vertex in the graphic.}
  \label{cuspfig}
  \end{center}
\end{figure}

The level set shown on the left is tangent to the level planes of $a$ at two points, which appear as a saddle and a central singularity in the surface.  These two points of tangency are points of $\mathcal{J}$.  As the level surfaces sweep through $N$, these tangent points form arcs of $\mathcal{J}$ that approach each other until they connect at the origin, shown in the middle surface in the figure.  After this, there are no tangencies between the level surfaces of $a$ and $b$.  Note that of the two arcs of $\mathcal{J}$ that meet at the cusp point, one is an edge of definite fold points and the other is an edge of indefinite fold points.  

Theorem~\ref{critsclassthm} states that we can assume the cusp is not a double point of the map from $\mathcal{J}$ into $\mathbf{R}$ so we can choose $N$ such that no double points of $\mathcal{J}$ are in $N$.  The local picture in $C$ and $\mathbf{R}^2$ is shown in Figure~\ref{cusp2fig}.
\begin{figure}[htb]
  \begin{center}
  \includegraphics[width=3.5in]{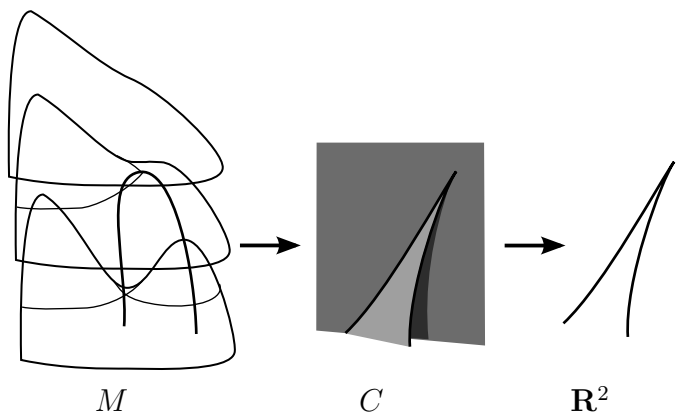}
  \put(-220,-15){$M$}
  \put(-120,-15){$C$}
  \put(-40,-15){$\mathbf{R}^2$}
  \caption{The local maps from $M$ to $C$ to $\mathbf{R}^2$ at a cusp.}
  \label{cusp2fig}
  \end{center}
\end{figure}

\section{Reading the graphic}
\label{linpathsect}

In order to find the sequence of stabilizations and destabilizations needed to prove Theorem~\ref{mainthm}, we would like to analyze the path of smooth functions defined by projecting a stable function $F \in C^\infty(M,\mathbf{R})$ orthogonally onto lines through the origin.  We can understand projections in general by looking at the projection onto a specific axis, then generalize to the others by ``rotating'' $F$ (i.e. composing $F$ with a rotation of the plane).  

We will consider the projection $p_y : \mathbf{R}^2 \rightarrow \mathbf{R}$ given by $p_y(x,y) = y$.  As noted above, if our stable function is defined as $F = f \times g$ then $p_y \circ F$ is the function $g$.  We would like to read information about $g$ from the graphic.  We will start with finding critical points.

\begin{Lem}
\label{horizcritlem}
If there are no horizontal tangents at cusps of the graphic $F(\mathcal{J}) \subset \mathbf{R}^2$ then there is a one-to-one correspondence between critical points of $f = p_x \circ F$ and points in the graphic at which there is a horizontal tangent.
\end{Lem}

\begin{proof}
The stable function $F$ determines a map from the tangent space $T_p M$ at each point $p \in M$ to the tangent space $T_{F(p)} \mathbf{R}^2$ of a point in the plane.  For a point in the complement of the discriminant set $\mathcal{J}$, this map (a homomorphism of vector spaces) has two dimensional image and one dimensional kernel.  By definition, this map has a two or three dimensional kernel at each point in $\mathcal{J}$.

By the classification of singularities described in Theorem~\ref{critsclassthm}, the kernel at each point has dimension two, and therefore the image of $T_p(M)$ is a one dimensional subspace of $T_{F(p)} \mathbf{R}^2$.  In fact, the image is the one dimensional subspace tangent to the edge of the graphic at $p$.

The function $g$ is the composition of $F$ with the projection $p_x$ onto the vertical axis.  The critical points of $g$ are the points $p \in M$ where the map from $T_p M$ to $T_{g(p)}\mathbf{R}$ has zero dimensional image.  Thus $T_p M$ must be mapped into the kernel in $T_{F(p)} \mathbf{R}^2$ of the map into $T_{g(p)} \mathbf{R}$.  This is the case precisely when $T_p M$ is mapped to a line perpindicular to the vertical axis, i.e. when the tangent is horizontal.
\end{proof}

For a more intuitive explanation of Lemma~\ref{horizcritlem}, consider the local picture:  Note that the slope of an arc of the graphic is precisely the ratio of the lengths of the gradient vectors of $f$ and $g$ at the corresponding point in the discriminant set.  (The gradients are parallel so the ratio of their lengths is well defined.)  As an arc of the discriminant set passes through a critical point of $g$, the gradient of $g$ goes to zero, while the gradient of $f$ is non-zero at every point.  In the graphic, this means that the slope of the corresponding arc goes to zero, so there is a horizontal tangency.

Near a horizontal tangency of the graphic, the edge can be identified with the graph of a unique function from $\mathbf{R}$ to $\mathbf{R}$, allowing us to define the second derivative of the graph at that point as the second derivative of this function.

\begin{Lem}
\label{2nderivlem}
If a point $p \in \mathcal{J}$ is a critical point of $g$ and is not a cusp point of $F$ then $p$ is non-degenerate if and only if the second derivative of the image in $\mathbf{R}^2$ of the arc through $p$ is non-zero.
\end{Lem}

The determinant of the Hessian at a critical point can be calculated from the models for points of $\mathcal{J}$ in Section~\ref{reebsect}.  The Lemma follows immediately from this.  We will leave this calculation to the reader and later present a more intuitive argument for why a degenerate critical point must appear at a horizontal inflection point.

We have so far carefully avoided analyzing the situation when the slope at a cusp is horizontal.  We will eventually deal with this case, but for now note that there are a finite number of cusps in $\mathcal{J}$ and therefore a finite number of slopes at which a cusp is horizontal.  This allows us to determine when the arc in $C^\infty(M,\mathbf{R})$ is generic, i.e. passes through a finite number of non-Morse functions.  We will say that a straight line in $\mathbf{R}^2$ is \textit{doubly tangent} to the graphic if it is tangent to the graphic at more than one point.

\begin{Lem}
\label{finitenonmorselem}
The path from $g$ to $f$ given by $\phi_t = cos(t) g + sin(t) f$ will pass through finitely many non-Morse functions if there are finitely many points in $\mathbf{R}^2$ at which the second derivative of the graphic $F(\mathcal{J})$ is zero and only finitely many straight lines in $\mathbf{R}^2$ are doubly tangent to the graphic.
\end{Lem}

\begin{proof}
The function $\phi_t$ is equal to the composition of $F$ with a projection onto an axis of $\mathbf{R}^2$ at angle $t$ clockwise from the vertical axis.  Equivalently, we can recover $\phi_t$ by rotating $F$ by angle $t$ counterclockwise, then composing with an orthogonal projection onto the vertical axis.  By Lemma~\ref{horizcritlem}, this composition (and therefore $\phi_t$) will be Morse if (1) there are no horizontal cusps, (2) at every horizontal tangency of each edge, the second derivative of the edge is non-zero and (3) any two horizontal tangencies project to distinct points in $\mathbf{R}$.

Because there are finitely many cusps, there are a finite number of rotation angles where there is a cusp with a horizontal tangency.  If there are only finitely many points in the edges of the graphic at which the second derivative is zero, then there will be a finite number of angles at which there is a horizontal inflection point.  Finally, if there are finitely many straight lines that are tangent to the graphic at two points then there will be a finite number of angles (given by the slopes of these lines) where two critical points are at the same level.  Thus if the assumptions of the Lemma are satisfied then there will be finitely many angles $t$ such that $\phi_t$ is non-Morse.
\end{proof}

To prove Theorem~\ref{mainthm}, we must show that when $t$ passes through the finite number of values for which $\phi_t$ is not Morse, the isotopy class of $\phi_t$ changes in a way that corresponds to a stabilization or destabilization in the induced Heegaard splitting at an inflection point or type two cusp, and does not change the Heegaard splitting otherwise.  

By Lemma~\ref{horizcritlem}, we can read the number of critical points (and therefore the genus) at each stage by looking at the number of horizontal tangencies.  As we rotate an inflection point or type two cusp, we see that the number of horizontal tangencies increases or decreases by two, as in Figure~\ref{rotinflectfig}.  Depending on the type of critical points that are created or removed, this either increases or decreases the genus by one or has no effect on the genus.  We must check that in the case when the genus changes, the new Heegaard splitting is a stabilization or destabilization of the original.

We will think of the Reeb complex as a union of Reeb graphs as follows:  The pre-image in the stable function $F$ of a generic line $\mathbf{R} \times \{y\} \subset \mathbf{R}^2$ is a (possible disconnected) surface $\Sigma_y$.  The restriction of $F$ to $\Sigma_y$ is a Morse function $f_y$.

The Reeb graph $R_y$ of $\Sigma_y$ is contained in the Reeb complex $C$ of $F$.  In particular, it is the preimage of $\mathbf{R} \times \{y\}$ in the map from $C$ to $\mathbf{R}^2$.  This is true for each $y \in \mathbf{R}$, so we can think of $C$ as the union of the Reeb graphs of the functions $f_y$.  If $F$ is the product $f \times g$ of Morse functions on $M$ then these horizontal slices of the Reeb complex are Reeb graphs for the restriction of $f$ to the level sets of $g$.

For a given $y \in \mathbf{R}$, the Euler characteristic of the surface $\Sigma_y$ is twice the Euler characteristic of the Reeb graph $R_y$ at $y$.  We can calculate this Euler characteristic as follows:  Each intersection of $\mathbf{R} \times \{y\}$ with an edge of definite fold points corresponds to a valence one vertex in $R_y$ and each intersection with an edge of indefinite fold points corresponds to a valence three vertex.

If $\mathbf{R} \times \{y\}$ intersects $n$ definite fold points and $m$ indefinite fold points then $R_y$ has $n + m$ vertices and $\frac{1}{2}n + \frac{3}{2}m$ edges so its Euler characteristic is $\frac{m - n}{2}$.  The Euler characteristic of $\Sigma_y$ is $m - n$.  As $t$ passes through a value where there is a horizontal tangent, the number of intersections with one type of edge increases or decreases by two.  

At a horizontal edge of definite fold points, two valence one vertices are added or removed.  Depending on whether the 2-cell is above or below the edge, this either adds or removes a sphere component of $\Sigma_y$ or increases the genus of a component.  At a horizontal edge of indefinite fold points, the genus of a component increases or decreases by one.  As we would hope, this is exactly the behavior of the level sets of a Morse function when they pass through a level containing an appropriate critical point.

We can now analyze how Morse functions induced by projecting $F$ onto different axes change at the non-generic angles, i.e. the angles at which there is a horizontal cusp or a horizontal inflection point.  We will look at a local model of each and consider how the level sets change as the non-generic point is rotated through a horizontal position.  Recall that there are three situations in which the function $\phi_t$ (constructed by rotating the graphic by angle $t$ and projecting onto the vertical axis) may be non-Morse: When there is a horizontal inflection point, when there are two horizontal tangencies at the same level and when there is a horizontal cusp.

If we rotate the local model of an inflection point through an angle where the inflection point is horizontal, we see that the number of horizontal tangencies either increases or decreases by two, as in Figure~\ref{rotinflectfig}.  There are eight cases to consider, defined by whether the second derivative changes from positive to negative or negative to positive, whether the edge is an edge of definite folds or indefinite folds, and whether there are more sheets of the Reeg complex above or below the edge.
\begin{figure}[htb]
  \begin{center}
  \includegraphics[width=3in]{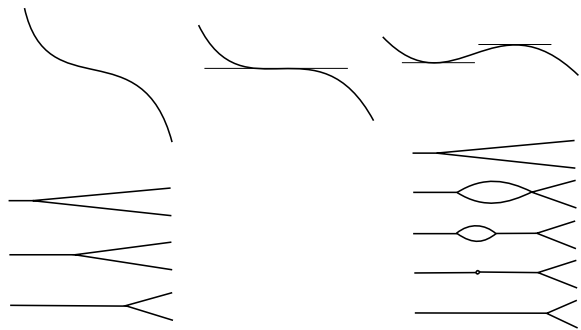}
  \caption{Rotating an inflection point through a horizontal position creates or removes two horizontal tangencies.  The Reeb graphs of horizontal slices before and after the rotation are shown below.}
  \label{rotinflectfig}
  \end{center}
\end{figure}

In the four cases when the edge consists of definite fold points, one of the critical points created or removed by the rotation has index zero or three and the second has index one or two, respectively.  Such a change to the Morse function does not change the isotopy class of the Heegaard splitting.  It changes the descending spine by adding or removing a trivial edge and a valence one vertex.

In the four cases when the edge consists of indefinite fold points, one of the critical points created or removed by the rotation has index one and the other has index two.  For one of these cases, the Reeb graphs for the level surfaces, defined by level slices of the Reeb complex, are shown in the bottom half of Figure~\ref{rotinflectfig}.  The corresponding level surfaces are shown in Figure~\ref{levelstabfig}.  The one-handle and two-handle define a stabilization in the induced Heegaard splitting.  Similar analysis of the other three cases shows that they also induce a single stabilization or destabilization in the induced Heegaard splittings.  (This is left to the reader.)
\begin{figure}[htb]
  \begin{center}
  \includegraphics[width=4.5in]{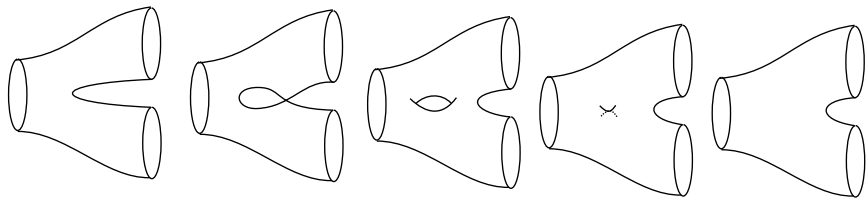}
  \caption{The level surface near an inflection point after it has been rotated through a horizontal angle.}
  \label{levelstabfig}
  \end{center}
\end{figure}

When $t$ passes through an angle where there are two or more critical points at the same level, the isotopy class of the induced Heegaard splitting does not change.  To see this, recall that by Lemma~\ref{kisaspinelem}, a spine of a handlebody (which determines the isotopy class of the Heegaard splitting) is given by a graph of descending edges in the Morse function.  When two critical points pass through the same level, the isotopy class of this graph does not change, so the isotopy class of the induced Heegaard splitting does not change.

The final case to consider is when $t$ passes through an angle where there is a horizontal cusp.  When $t$ passes through an angle where a type one cusp becomes horizontal, the number of horizontal tangencies does not change.  Whether or not the critical point at the cusp is degenerate when the cusp is horizontal, passing through the angle where the cusp is horizontal does not change the isotopy classes of the surfaces near the cusp and therefore does not change the isotopy class of the induced Heegaard splitting.  Figure~\ref{rotcusp1fig} shows how the Reeb graphs of the level sets change when the upper edge is a simple fold edge and the cusp points to the right.  The other cases are similar.
\begin{figure}[htb]
  \begin{center}
  \includegraphics[width=3in]{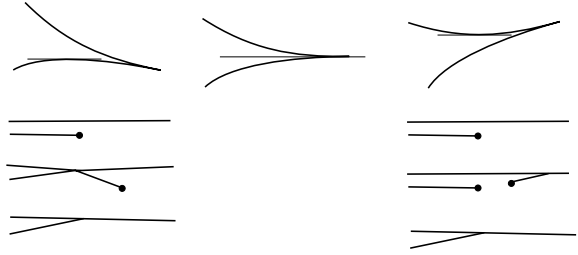}
  \caption{Rotating a type one cusp through a horizontal position replaces a horizontal tangency in one of the edges with a horizontal tangency in the other edge.  In the Reeb graphs for the level sets shown below, we assume that the upper edge is a definite fold edge.}
  \label{rotcusp1fig}
  \end{center}
\end{figure}

When $t$ passes through an angle where a type two cusp becomes horizontal, the number of horizontal tangencies either increases or decreases by two.  Because the Morse function changes, the critical point when the cusp is horizontal must be degenerate.  Figure~\ref{rotcusp2fig} shows how the level sets change in the case when the cusp is concave up, points to the right and the upper edge is a simple fold edge.  The other configurations of type two cusps are similar, and all correspond to a singl stabilization or destabilization of the Heegaard splitting.
\begin{figure}[htb]
  \begin{center}
  \includegraphics[width=3in]{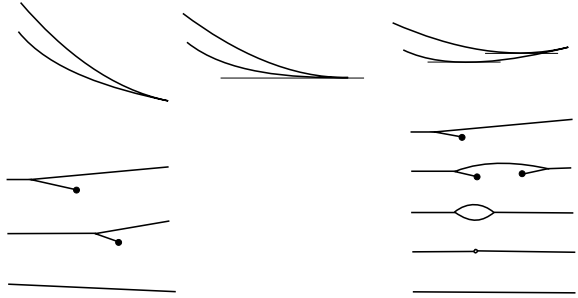}
  \caption{Rotating a type two cusp through a horizontal position creates or removes two horizontal tangencies.  In the Reeb graphs shown below, we assume the upper edge is a definite fold edge.}
  \label{rotcusp2fig}
  \end{center}
\end{figure}

\section{Proof of Theorem~\ref{mainthm}}
\label{proofsect}

Let $(\Sigma, H_1, H_2)$ and $(\Sigma', H'_1, H'_2)$ be Heegaard splittings of a closed 3-manifold $M$.  Construct Morse functions $f$ and $g$ such that $\Sigma$ and $\Sigma'$ are induced by $f$ and $g$ respectively.  Define $\phi_t = cos(t) g + sin(t) f$.  In order to prove Theorem~\ref{mainthm}, we will show first that the path defines a sequence of stabilizations and destabilizations and second that the number of steps in the sequence is less than or equal to the number of negative slope inflection points and type two cusps.

\begin{proof}[Proof of Theorem~\ref{mainthm}]
Isotope $f$ and $g$ so that in the graphic (i.e. the image in $\mathbf{R}^2$ of the discriminant set of $F = f \times g$), there are a finite number of points where the second derivative of the function defined by the edge is zero and finitely many doubly tangent straight lines.  There are then a finite number of angles $t_1 < \dots < t_n \frac{\pi}{2}$ such that rotating the graphic counter-clockwise by angle $t_i$ creates a horizontal inflection point, a horizontal cusp or two horizontal tangents at the same level.

By Lemma~\ref{2nderivlem}, for $t_i < t < t_{i+1}$, the function $\phi_t$ is a Morse function.  Because the arc $(t_i,t_{i+1})$ is contained in the set of Morse functions, any two functions in the arc are isotopic and induce isotopic Heegaard splittings of $M$.

The Heegaard splittings induced by functions in the arc $[0,t_1)$ are isotopic to $(\Sigma', H'_1, H'_2)$.  If rotating the graphic by angle $t_1$ produces a horizontal inflection point in a definite fold edge, a horizontal type one cusp or two horizontal tangents at the same level then, as we saw in Section~\ref{linpathsect}, the induced Heegaard splittings in $(t_1,t_2)$ are also isotopic to $(\Sigma', H'_1, H'_2)$.  If at angle $t_1$, there is a horizontal inflection point in an indefinite fold edge or a type two cusp then the Heegaard splittings induced by functions in the arc $(t_1,t_2)$ is a single stabilization or destabilization of $(\Sigma',H'_1,H'_2)$.

By repeating this argument at each angle $t_i$, we find a sequence of stabilizations and destabilizations.  The last Heegaard splitting, induced by functions in the arc $(t_n,\frac{\pi}{2}]$, is $(\Sigma,H_1,H_2)$.  Each step in the sequence corresponds to a negative slope inflection point or a negative slope type two cusp so by Lemma~\ref{ifcstabslem}, there is a common stabilization of genus $(p + q + c)/2$ where $p$ and $q$ are the genera of $\Sigma$ and $\Sigma'$, respectively and $c$ is the number of negative slope inflection points in indefinite fold edges and type two cusps.
\end{proof}

\bibliographystyle{amsplain}
\bibliography{inflects}

\end{document}